\documentclass[reqno]{amsart}
\usepackage{amssymb}
\usepackage{graphicx}

\usepackage[usenames, dvipsnames]{color}
\usepackage{verbatim}
\usepackage{mathrsfs}

\numberwithin{equation}{section}

\newtheorem{theorem}{Theorem}[section]

\newtheorem{lemma}[theorem]{Lemma}
\newtheorem{proposition}[theorem]{Proposition}

\theoremstyle{definition}
\newtheorem{remark}[theorem]{Remark}

\theoremstyle{definition}

\theoremstyle{definition}

\makeatletter
\def\dashint{\operatorname%
{\,\,\text{\bf-}\kern-.98em\DOTSI\intop\ilimits@\!\!}}
\makeatother

\def\\det{\text{det}}

\def\.5{\frac{1}{2}}

\def\bR{\mathbb{R}}

\def\bN{\mathbb{N}}

\def\cA{\mathcal{A}}

\def\cP{\mathcal{P}}
\def\cM{\mathcal{M}}

\newcommand{\RN}[1]{%
  \textup{\uppercase\expandafter{\romannumeral#1}}%
}

\renewcommand{\epsilon}{\varepsilon}

\begin{document}
\title[Nonlocal elliptic equations]{Dini estimates for nonlocal fully nonlinear elliptic equations}
\author[H. Dong]{Hongjie Dong}
\address[H. Dong]{Division of Applied Mathematics, Brown University,
182 George Street, Providence, RI 02912, USA}
\email{Hongjie\_Dong@brown.edu}
\thanks{H. Dong was partially supported by the NSF under agreements DMS-1056737 and DMS-1600593.}

\author[H. Zhang]{Hong Zhang}
\address[H. Zhang]{Division of Applied Mathematics, Brown University,
182 George Street, Providence, RI 02912, USA}
\email{Hong\_Zhang@brown.edu}
\thanks{H. Zhang was partially supported by the NSF under agreement DMS-1056737.}

\begin{abstract}
We obtain Dini type estimates for a class of concave fully nonlinear nonlocal elliptic equations of order $\sigma\in (0,2)$ with rough and non-symmetric kernels. The proof is based on a novel application of Campanato's approach and a refined $C^{\sigma+\alpha}$ estimate in \cite{DZ16}.
\end{abstract}
\maketitle

\section{Introduction and main results}
The paper is a continuation of our previous work \cite{DZ16}, where we studied  Schauder estimates for concave fully nonlinear nonlocal elliptic and parabolic equations. In particular,  when the kernels are translation invariant and the data are merely bounded and measurable, we proved the $C^{\sigma}$ estimate, which is very different from the classical theory for second-order elliptic and parabolic equations. In this paper, we consider concave fully nonlinear nonlocal elliptic equations with Dini continuous coefficients and nonhomogeneous terms, and establish a $C^\sigma$ estimate under these assumptions.

The study of classical elliptic equations with Dini continuous coefficients and data has a long history. Burch \cite{Burch78} first considered divergence type linear elliptic equations with Dini continuous coefficients and data, and estimated the modulus of continuity of the derivatives of solutions. The corresponding result for concave fully nonlinear elliptic equations was obtained by Kovats \cite{Kovats97}, which generalized a previous result by Safonov \cite{Saf88}. Wang \cite{Wang06} studied linear non-divergence type elliptic and parabolic equations with Dini continuous coefficients and data, and gave a simple proof to estimate the modulus of continuity of the second-order derivatives of solutions. See, also \cite{Lieb87,Sp81,Bao02,DG02,MM11,Y.Li2016}, and the references therein.

Recently, there is extensive work on the regularity theory for nonlocal elliptic and parabolic equations. For example, $C^\alpha$ estimates, $C^{1,\alpha}$ estimates, Evans-Krylov type theorem, and Schauder estimates were established in the past decade. See, for instance, \cite{CS09,CS11,DK11,DK13,KL13,CD14,CK15, JX15,JX152,JS15,Mou16}, and the references therein.  In particular, Mou \cite{Mou16} investigated a class of concave fully nonlinear nonlocal elliptic equations with smooth symmetric kernels, and obtained the $C^{\sigma}$ estimate under a slightly stronger assumption than the usual Dini continuity on the coefficients and data.  The author implemented a recursive Evans-Krylov theorem, which was first studied by Jin and Xiong \cite{JX15}, as well as a perturbation type argument.  In this paper, by using a novel perturbation type argument, we relax the regularity assumption to simply Dini continuity and also remove the symmetry and smoothness assumptions on the kernels.

To be more specific, we are interested in fully nonlinear nonlocal elliptic equations in the form

\begin{equation}
                                    \label{eq 1}
\inf_{\beta\in \cA}(L_{\beta}u+f_{\beta})=0,
\end{equation}
where $\cA$ is an index set and for each $\beta\in \cA$,
$$
L_{\beta}u=\int_{\bR^d} \delta u(x,y)K_{\beta}(x,y)\,dy,
$$
\begin{align*}
\delta u(x,y)=
\begin{cases}
u(x+y)-u(x)-y\cdot Du(x)\quad&\text{for}\,\, \sigma\in(1,2),\\
u(x+y)-u(x)-y\cdot Du(x)\chi_{B_1}\quad&\text{for}\,\, \sigma=1,\\
u(x+y)-u(x)\quad &\text{for}\,\, \sigma\in (0,1),
\end{cases}
\end{align*}
and
$$
K_{\beta}(x,y)=a_\beta(x,y)|y|^{-d-\sigma}.
$$
This type of nonlocal operators was first investigated by Komatsu  \cite{Komatsu84}, Mikulevi$\check{\text{c}}$ius and Pragarauskas \cite{MP92,MP14}, and later by Dong and Kim \cite{DK11,DK13}, and Schwab and Silvestre \cite{SS16}, to name a few.

We assume that $a(\cdot,\cdot)\in [\lambda, \Lambda]$ for some ellipticity constants $0<\lambda\le \Lambda$, and is merely measurable with respect to the $y$ variable. When $\sigma=1$, we additionally assume that
\begin{equation}
\int_{S_r}yK_{\beta}(x,y)\,ds=0,\label{eq10.58}
\end{equation}
for any $r>0$, where $S_r$ is the sphere of radius $r$ centered at the origin.
We say that a function $f$ is Dini continuous if its modulus of continuity $\omega_f$ is a Dini function, i.e.,
$$
\int_0^1 \omega_f(r)/r\,dr<\infty.
$$

The following theorem is our main result.

\begin{theorem}\label{thm 1}
Let $\sigma\in (0,2)$, $0<\lambda\le \Lambda<\infty$, and $\cA$ be an index set. Assume for each $\beta\in \cA$, $K_{\beta}$ satisfies \eqref{eq10.58} when $\sigma=1$, and
\begin{align*}
&\big|a_{\beta}(x,y)-a_{\beta}(x',y)\big|
\le \Lambda\omega_a(|x-x'|),\\
&|f_{\beta}(x)-f_{\beta}(x')|\le \omega_f(|x-x'|),\quad
\sup_{\beta\in \mathcal{A}}\|f_{\beta}\|_{L_\infty(B_1)}<\infty,
\end{align*}
where $\omega_a$ and $\omega_f$ are Dini functions.
Suppose $u\in C^{\sigma^+}(B_1)$ is a solution of \eqref{eq 1} in $B_1$ and is Dini continuous in $\bR^d$.
Then we have the a priori estimate
\begin{equation}\label{eq12.17}
[u]_{\sigma;B_{1/2}}\le C\|u\|_{L_\infty}+C\sup_\beta\|f_\beta\|_{L_\infty(B_1)}
+C\sum_{j=1}^\infty\big(\omega_u(2^{-j})+\omega_f(2^{-j})\big)
\end{equation}
where $C>0$ is a constant depending only on $d$, $\sigma$, $\lambda$, $\Lambda$, and $\omega_a$
Moreover, when $\sigma\neq 1$, we have
$$
\sup_{x_0\in B_{1/2}} [u]_{\sigma;B_r(x_0)}\to 0 \quad\text{as}\quad r\to 0
$$
with a decay rate depending only on $d$, $\sigma$, $\lambda$, $\Lambda$, $\omega_a$, $\omega_f$, $\omega_u$, and $\sup_{\beta\in \cA}\|f_\beta\|_{L_\infty(B_1)}$.
When $\sigma=1$,
$Du$ is uniformly continuous in $B_{1/2}$ with a modulus of continuity controlled by the quantities before.
\end{theorem}
Here for simplicity we assume $u\in C^{\sigma^+}(B_1)$, which means that $u\in C^{\sigma+\varepsilon}(B_1)$ for some arbitrary $\varepsilon>0$. This condition is only needed for $L_\beta u$ to be well defined, and may be replaced by other weaker conditions.
\begin{remark}
By a careful inspection of the proofs below, one can see that the estimates above in fact only depend on $d$, $\sigma$, $\lambda$, $\Lambda$, $\sup_{\beta\in \cA}\|f_\beta\|_{L_\infty(B_1)}$, the modulus continuity $\omega_f$ of $f_\beta$ in $B_1$, $\omega_a(r)$, $\omega_u(r)$ for $r\in (0,1)$, and $\|u\|_{L_{1,w}}$, where the weight $w=w(x)$ is equal to $(1+|x|)^{-d-\sigma}$.
In particular, $u$ does not need to be globally bounded in $\bR^d$.
\end{remark}
Roughly speaking,  the proof can be divided into two steps: We first show that Theorem \ref{thm 1} holds when the equation is satisfied in the whole space; Then we implement a localization argument to treat the general case.  In Step one, our proof is based on a refined $C^{\sigma+\alpha}$ estimate in our previous paper \cite{DZ16} and a new perturbation type argument, as the standard perturbation techniques do not seem to work here. The novelty of this method is that instead of estimating $C^\sigma$ semi-norm of the solution, we construct and bound certain semi-norms of the solution, see Lemmas \ref{lem2.3} and \ref{lem2.4}. When $\sigma<1$, such semi-norm is defined as a series of lower-order H\"older semi-norms of $u$. This is in the spirit of Campanato's approach first developed in \cite{Ca66}. Heuristically, in order for the nonlocal operator to be  well defined, the solution needs to be smoother than $C^\sigma$. To resolve this problem, we divide the integral domain into annuli, which allows us to use a lower-order semi-norm to estimate the integral in each annulus.  The series of lower-order semi-norms, which turns out to be slightly stronger than the $C^\sigma$ semi-norm, further implies that
$$
[u]_{\sigma;B_r(x_0)}\rightarrow 0\quad\text{as} \quad r\rightarrow 0
$$
uniformly in $x_0$.
In particular,  when $\sigma = 1$ we are able to estimate the modulus of continuity of the gradient of solutions. The proof of the case when $\sigma\ge 1$ is more difficult than that of the case when $\sigma<1$. This is mainly due to the fact that the series of lower-order H\"older semi-norms of the solution itself is no longer sufficient to estimate the $C^\sigma$ norm. Therefore, we need to subtract  a polynomial from the solution  in the construction of the semi-norm. In some sense, the polynomial should be taken to minimize the series.
It turns out that when $\sigma>1$, up to a constant we can choose the polynomial to be the first-order Taylor's expansion of the solution. The case $\sigma= 1$ is particularly challenging since the polynomial needs to be selected carefully, for which an additional mollification argument is applied.

The organization of this paper is as follows. In the next section, we introduce some notation and preliminary results that are necessary in the proof of our main theorem. Some of these results might be of independent interest. In section 3, we first prove a global version of Theorem \ref{thm 1} and then localize the result to obtain Theorem \ref{thm 1}.
\section{Preliminaries}

We will frequently use the following identity
\begin{align}
                            \label{eq10.22}
&2^j\big(u(x+2^{-j}\ell)-u(x)\big)-\big(u(x+\ell)-u(x)\big)\nonumber\\
&=\sum_{k=1}^j 2^{k-1}\big(2u(x+2^{-k}\ell)-u(x+2^{-k+1}\ell)-u(x)\big),
\end{align}
which holds for any $\ell\in \bR^d$ and nonnegative integer $j$.

Denote $\cP_1$ to be the set of first-order polynomials of $x$.
\begin{lemma}
                            \label{lem2.3}
Let $\alpha\in (0,\sigma)$ be a constant.

(i) When $\sigma\in (0,1)$, we have
\begin{equation}
                            \label{eq10.08}
[u]_{\sigma}\le C\sup_{r>0}\sup_{x_0\in \bR^{d}} r^{\alpha-\sigma}[u]_{\Lambda^\alpha(B_r(x_0))}
\le C\sup_{r>0}\sup_{x_0\in \bR^{d}} r^{\alpha-\sigma}[u]_{\alpha;B_r(x_0)},
 \end{equation}
where $C>0$ is a constant depending only on $d$, $\alpha$, and $\sigma$.

(ii) When $\sigma\in (1,2)$, we have
\begin{equation}
                                        \label{eq10.09}
[u]_{\sigma}\le C\sup_{r>0}\sup_{x_0\in \bR^{d}} r^{\alpha-\sigma}[u]_{\Lambda^\alpha(B_r(x_0))}
\le C\sup_{r>0}\sup_{x_0\in \bR^{d}} r^{\alpha-\sigma}\inf_{p\in \mathcal{P}_1}[u-p]_{\alpha;B_r(x_0)},
 \end{equation}
where $C>0$ is a constant depending only on $d$, $\alpha$, and $\sigma$.

(iii) When $\sigma=1$, we have
\begin{align}
\|Du\|_{L_\infty}&\le C\sum_{k=0}^\infty\sup_{x_0\in \bR^{d}} 2^{-k(\alpha-1)}[u]_{\Lambda^{\alpha}(B_{2^{-k}}(x_0))}
+C\sup_{\substack{x,x'\in \bR^d\\|x-x'|=1}}|u(x)-u(x')|\nonumber\\
                                            \label{eq10.10}
&\le C\sum_{k=0}^\infty\sup_{x_0\in \bR^{d}} 2^{-k(\alpha-1)}\inf_{p\in \mathcal{P}_1}[u-p]_{\alpha;B_{2^{-k}}(x_0)}+C\sup_{\substack{x,x'\in \bR^d\\|x-x'|=1}}|u(x)-u(x')|,
 \end{align}
where $C>0$ is a constant depending only on $d$ and $\alpha$. Moreover, we can estimate the modulus of continuity of $Du$ by the remainder of the summation on the right-hand side of \eqref{eq10.10}.
\end{lemma}
\begin{proof}
First we consider the case when $\sigma\in (0,1)$. Let $x,x'\in \bR^d$ be  two different points. Denote $h=|x-x'|$. Since
$$
u(x')-u(x)=\frac 1 2 \big(u(2x'-x)-u(x)\big)-\frac 1 2\big(u(2x'-x)-2u(x')+u(x)\big),
$$
we get
\begin{align*}
&h^{-\sigma}|u(x')-u(x)|\\
&\le 2^{\sigma-1} (2h)^{-\sigma}\big(u(2x'-x)-u(x)\big)+h^{-\sigma}|u(2x'-x)-2u(x')+u(x)|\\
&\le 2^{\sigma-1} (2h)^{-\sigma}\big(u(2x'-x)-u(x)\big)+\sup_{x\in \bR^{d}} h^{\alpha-\sigma}[u]_{\Lambda^\alpha(B_h(x))}.
\end{align*}
Taking the supremum with respect to $x$ and $x'$ on both sides, we get
$$
[u]_\sigma\le 2^{\sigma-1}[u]_\sigma+\sup_{x\in \bR^{d}} h^{\alpha-\sigma}[u]_{\Lambda^\alpha(B_h(x))},
$$
which together with the triangle inequality gives \eqref{eq10.08}.

For $\sigma\in (1,2)$, let $\ell\in\bR^d$ be a unit vector and $\varepsilon\in (0,1/16)$ be a small constant to be specified later. For any two distinct points $x,x'\in \bR^d$, we denote $h=|x-x'|$. By the triangle inequality,
\begin{equation}
                        \label{eq8.43}
h^{1-\sigma}|D_\ell u(x)-D_\ell u(x')|
\le I_1+I_2+I_3,
\end{equation}
where
\begin{align*}
I_1&=h^{1-\sigma}|D_\ell u(x)-(\varepsilon h)^{-1}(u(x+\varepsilon h\ell)-u(x))|,\\
I_2&=h^{1-\sigma}|D_\ell u(x')-(\varepsilon h)^{-1}(u(x'+\varepsilon h\ell)-u(x'))|,\\
I_3&=h^{1-\sigma}(\varepsilon h)^{-1}|(u(x+\varepsilon h\ell)-u(x))-(u(x'+\varepsilon h\ell)-u(x'))|.
\end{align*}
By the mean value theorem,
\begin{equation}
                                    \label{eq3.02b}
I_1+I_2\le 2\varepsilon^{\sigma-1} [Du]_{\sigma}.
\end{equation}
Now we choose and fix a $\varepsilon$ sufficiently small depending only on $\sigma$ such that $2\varepsilon^{\sigma-1}\le 1/2$.
Using the triangle inequality, we have
\begin{align*}
I_3 \le Ch^{-\sigma}\big(|u(x+\varepsilon h\ell)+u(x')-2u(\bar x)|+|u(x'+\varepsilon h\ell)+u(x)-2u(\bar x)|\big),
\end{align*}
where $\bar x=(x+\varepsilon h\ell+x')/2$.
Thus,
\begin{equation}
                                        \label{eq3.03b}
I_3\le Ch^{\alpha-\sigma}[u]_{\Lambda^\alpha(B_{h}(\bar x))}.
\end{equation}
Combining \eqref{eq8.43}, \eqref{eq3.02b}, and \eqref{eq3.03b}, we get \eqref{eq10.09} as before.

Finally, we treat the case when $\sigma=1$. 
It follows from \eqref{eq10.22} that
\begin{align*}
2^j\big|u(x+2^{-j}\ell)-u(x)\big|\le 2|u(x+\ell)-u(x)|+\sum_{k=1}^j 2^{-k(\alpha-1)}[u]_{\Lambda^\alpha(B_{2^{-k}}(x+2^{-k}\ell))}.
\end{align*}
Taking $j\to \infty$, we obtain the desired inequality. For the continuity estimate, let $\ell\in\bR^d$ be a unit vector. Assume that $|x-x'|\in [2^{-i-1},2^{-i})$ for some positive integer $i$.
From \eqref{eq10.22}, for any $j\ge i+1$,
\begin{align*}
&2^j\big(u(x+2^{-j}\ell)-u(x)\big)-2^i\big(u(x+2^{-i}\ell)-u(x)\big)\\
&=\sum_{k=i+1}^j 2^{k-1}\big(2u(x+2^{-k}\ell)-u(x+2^{-k+1}\ell)-u(x)\big)
\end{align*}
and a similar identity holds with $x'$ in place of $x$. Then we have
\begin{align*}
&|D_\ell u(x)-D_\ell u(y)|=\lim_{j\to \infty}\Big|2^j\big(u(x+2^{-j}\ell)-u(x)\big)
-2^j\big(u(x'+2^{-j}\ell)-u(x')\big)\Big|\\
&\le \Big|2^i\big(u(x+2^{-i}\ell)-u(x)\big)
-2^i\big(u(x'+2^{-i}\ell)-u(x')\big)\Big|\\
&\quad +
\sum_{k=i+1}^\infty\sup_{x_0\in \bR^{d}} 2^{-k(\alpha-1)}[u]_{\Lambda^\alpha(B_{2^{-k}}(x_0))}.
\end{align*}
By the triangle inequality, the first term on the right-hand side is bounded by
$$
2^i|u(x+2^{-i}\ell)-2u(\bar x)+u(x')|+
2^i|u(x'+2^{-i}\ell)-2u(\bar x)+u(x)|
$$
with $\bar x=(x+2^{-i}+x')/2$, which is further bounded by
$$
2^{1+i(1-\alpha)}[u]_{\Lambda^\alpha(B_{2^{-i}}(\bar x))}.
$$
Therefore,
$$
|D_\ell u(x)-D_\ell u(y)|\le C\sum_{k=i}^\infty\sup_{x_0\in \bR^{d}} 2^{-k(\alpha-1)}[u]_{\Lambda^\alpha(B_{2^{-k}}(x_0))},
$$
which converges to $0$ as $i\to \infty$ uniformly with respect to $\ell$. The lemma is proved.
\end{proof}

The following lemma will be used to estimate the error term in the freezing coefficient argument.
\begin{lemma}
                            \label{lem2.4}
Let $\alpha\in (0,1)$ and $\sigma\in (1,2)$ be constants.
Then for any $u\in C^1$, we have
\begin{equation}
                                        \label{eq8.32}
\sum_{k=0}^\infty 2^{k(\sigma-\alpha)}\sup_{x_0\in \bR^d}
[u-P_{x_0}u]_{\alpha;B_{2^{-k}(x_0)}}
\le C\sum_{k=0}^\infty 2^{k(\sigma-\alpha)}
\sup_{x_0\in \bR^d}\inf_{p\in \cP_1}[u-p]_{\alpha;B_{2^{-k}(x_0)}}
\end{equation}
and
\begin{equation}
                                        \label{eq10.32}
\sum_{k=0}^\infty 2^{k\sigma}
\sup_{x_0\in \bR^d}\|u-P_{x_0}u\|_{L_\infty(B_{2^{-k}(x_0)})}
\le C\sum_{k=0}^\infty 2^{k(\sigma-\alpha)}\sup_{x_0\in \bR^d}
[u]_{\Lambda^{\alpha}(B_{2^{-k}(x_0)})},
\end{equation}
where $P_{x_0}u$ is the first-order Taylor expansion of $u$ at $x_0$, and $C>0$ is a constant depending only on $d$, $\alpha$, and $\sigma$.
\end{lemma}
\begin{proof}
Denote
$$
b_k:=2^{k(\sigma-\alpha)}
\sup_{x_0\in \bR^d}\inf_{p\in \cP_1}[u-p]_{\alpha;B_{2^{-k}}(x_0)}.
$$
Then for any $x_0\in \bR^d$ and each $k=0,1,\ldots$, there exists $p_k\in \cP_1$ such that
$$
[u-p_k]_{\alpha;B_{2^{-k}}(x_0)}\le 2b_k2^{-k(\sigma-\alpha)}.
$$
By the triangle inequality, for $k\ge 1$ we have
\begin{equation}
                                \label{eq8.45}
[p_{k-1}-p_k]_{\alpha;B_{2^{-k}}(x_0)}\le 2b_k2^{-k(\sigma-\alpha)}
+2b_{k-1}2^{-(k-1)(\sigma-\alpha)}.
\end{equation}
It is easily seen that
\begin{equation*}
[p_{k-1}-p_k]_{\alpha;B_{2^{-k}}(x_0)}=|\nabla p_{k-1}-\nabla p_k|2^{-(k-1)(1-\alpha)},
\end{equation*}
which together with \eqref{eq8.45} implies that
\begin{equation}
                            \label{eq8.47}
|\nabla p_{k-1}-\nabla p_k|\le C(b_k+b_{k-1})2^{-k(\sigma-1)}.
\end{equation}
Since $\sum_0^k b_k<\infty$, from \eqref{eq8.47} we see that $\{\nabla p_k\}$ is a Cauchy sequence in $\bR^d$. Let $q=q(x_0)\in \bR^d$ be its limit, which clearly satisfies for each $k\ge 0$,
$$
|q-\nabla p_k|\le C\sum_{j=k}^\infty 2^{-j(\sigma-1)}b_j.
$$
By the triangle inequality, we get
\begin{align}
                        \label{eq9.08}
&[u-q\cdot x]_{\alpha;B_{2^{-k}}(x_0)}
\le [u-p_k]_{\alpha;B_{2^{-k}}(x_0)}+[p_k-q\cdot x]_{\alpha;B_{2^{-k}}(x_0)}\nonumber\\
&\le C2^{-k(1-\alpha)}\sum_{j=k}^\infty 2^{-j(\sigma-1)}b_j\le C2^{-k(\sigma-\alpha)},
\end{align}
which implies that
$$
\|u-u(x_0)-q\cdot (x-x_0)\|_{L_\infty(B_{2^{-k}}(x_0))}\le C2^{-k\sigma},
$$
and thus $q=\nabla u(x_0)$. It then follows \eqref{eq9.08} that
\begin{align*}
&\sum_{k=0}^\infty 2^{k(\sigma-\alpha)}\sup_{x_0\in \bR^d}
[u-P_{x_0}u]_{\alpha;B_{2^{-k}}(x_0)}\le C\sum_{k=0}^\infty 2^{k(\sigma-1)}
\sum_{j=k}^\infty 2^{-j(\sigma-1)}b_j\\
&=C\sum_{j=0}^\infty 2^{-j(\sigma-1)}b_j
\sum_{k=0}^j 2^{k(\sigma-1)}\le C\sum_{j=0}^\infty b_j.
\end{align*}
This completes the proof of \eqref{eq8.32}.

Next we show \eqref{eq10.32}. For any $x\in B_{2^{-k}}$, it follows from \eqref{eq10.22} that for $j\ge 1$,
\begin{align*}
&u(x)-u(0)-2^j\big(u(2^{-j}x)-u(0)\big)\nonumber\\
&=\sum_{i=0}^{j-1} 2^{i}\big(u(2^{-i}x)+u(0)-2u(2^{-i-1}x)\big).
\end{align*}
Sending $j\to \infty$, we obtain
\begin{align*}
&\big|u(x)-u(0)-x\cdot \nabla u(0)\big|\le \sum_{i=0}^\infty 2^{i}\big|u(2^{-i}x)+u(0)-2u(2^{-i-1}x)\big|\\
&\le 2^{-\alpha}\sum_{i=0}^\infty 2^{i-(i+k)\alpha}[u]_{\Lambda^{\alpha}(B_{2^{-(k+i)}})}
=2^{-\alpha}\sum_{i=k}^\infty 2^{i-k-i\alpha}[u]_{\Lambda^{\alpha}(B_{2^{-i}})},
\end{align*}
where we shifted the index in the last equality.
Therefore, by shifting the coordinates and sum in $k$, we have
\begin{align*}
&\sum_{k=0}^\infty 2^{k\sigma}\sup_{x_0\in \bR^d}\|u-P_{x_0}u\|_{L_\infty(B_{2^{-k}})(x_0)}\\
&\le C \sum_{k=0}^\infty 2^{k(\sigma-1)}\sum_{i=k}^\infty 2^{i(1-\alpha)}\sup_{x_0\in \bR^d}[u]_{\Lambda^{\alpha}(B_{2^{-i}}(x_0))}\\
&= C\sum_{i=0}^\infty 2^{i(1-\alpha)}\sup_{x_0\in \bR^d}[u]_{\Lambda^{\alpha}(B_{2^{-i}}(x_0))}
\sum_{k=0}^i 2^{k(\sigma-1)}\\
&\le C\sum_{i=0}^\infty 2^{i(\sigma-\alpha)}\sup_{x_0\in \bR^d}[u]_{\Lambda^{\alpha}(B_{2^{-i}}(x_0))},
\end{align*}
where we switched the order of the summations in the second equality and in the last inequality we used the condition that $\sigma>1$.
The lemma is proved.
\end{proof}

Let $\zeta\in C_0^\infty(B_{1})$ be a nonnegative radial  function with unit integral.
For $R>0$, we define the mollification of a function $u$ by
$$
u^{(R)}(x)=\int_{\bR^{d}} u(x-Ry)\zeta(y)\,dy.
$$
The next lemmas will be used in the estimate of $M_j$ in Proposition \ref{prop1.1} when $\sigma= 1$.
\begin{lemma}
                            \label{lem2.1}
Let $\beta\in (0,1]$, $\alpha\in (0,1+\beta)$, and $0<R\le R_1<\infty$. Then for any
$u\in \Lambda^\alpha(B_{2R_1})$,
we have
\begin{equation}
                                            \label{eq11.15}
[Du^{(R)}]_{\beta;B_{R_1}}\le C(d,\beta,\alpha)
R^{\alpha-1-\beta}[u]_{\Lambda^\alpha(B_{2R_1})}.
\end{equation}
\end{lemma}
\begin{proof}
We begin by estimating $\|D_{\ell}^2u\|_{0;B_{R_1}}$ for a fixed unit vector $\ell\in \bR^d$. Because $D_{\ell}^2\zeta$ is even with respect to $x$ and has zero integral, using integration by parts we have for any $x\in B_{R_1}$,
\begin{align*}
&|D_{\ell}^2u^{(R)}(x)|=R^{-2}\Big|\int_{\bR^{d}} u(x-Ry)D_{\ell}^2\zeta(y)\,dy\Big|\\
&=\frac {R^{-2}}2\Big|\int_{\bR^{d}} \big(u(x-Ry)+u(x+Ry)-2u(x)\big)D_{\ell}^2\zeta(y)\,dy\Big|\\
&\le CR^{\alpha-2}[u]_{\Lambda^\alpha(B_{2R_1})}\int_{\bR^{d}} |y|^\alpha D_{\ell}^2\zeta(y)\,dy\le CR^{\alpha-2}[u]_{\Lambda^\alpha(B_{2R_1})}.
\end{align*}
Using the identity, $2D_{ij}u=2D_{\ell}^2u-D_i^2 u-D_j^2 u$, where $\ell=(e_i+e_j)/\sqrt 2$, we obtain the desired inequality \eqref{eq11.15} when $\beta=1$.

Next we consider the case when $\beta\in (0,1)$. We follow the proof of Lemma \ref{lem2.3}. Let $\ell\in\bR^d$ be a unit vector, and $\varepsilon\in (0,1/16)$ be a small constant to be specified later. For any two distinct points $x,x'\in B_{R_1}$, let $h=|x-x'|(<2R_1)$. It is easily seen that there exist two points $y\in B_{\varepsilon h}(x)\cap B_{R_1}$ and $y'\in B_{\varepsilon h}(x')\cap B_{R_1}$ such that
$$
y+\varepsilon h\ell\in B_{\varepsilon h}(x)\cap B_{R_1},\quad
y'+\varepsilon h\ell\in B_{\varepsilon h}(x')\cap B_{R_1}.
$$
By the triangle inequality,
$$
h^{-\beta}|D_\ell u^{(R)}(x)-D_\ell u^{(R)}(x')|
\le I_1+I_2+I_3,
$$
where
\begin{align*}
I_1&=h^{-\beta}|D_\ell u^{(R)}(x)-(\varepsilon h)^{-1}(u^{(R)}(y+\varepsilon h\ell)-u^{(R)}(y))|,\\
I_2&=h^{-\beta}|D_\ell u^{(R)}(x')-(\varepsilon h)^{-1}(u^{(R)}(y'+\varepsilon h\ell)-u^{(R)}(y'))|,\\
I_3&=h^{-\beta}(\varepsilon h)^{-1}|(u^{(R)}(y+\varepsilon h\ell)-u^{(R)}(y))-(u^{(R)}(y'+\varepsilon h\ell)-u^{(R)}(y'))|.
\end{align*}
By the mean value theorem,
\begin{equation}
                                    \label{eq3.02}
I_1+I_2\le 2\varepsilon^\beta [Du^{(R)}]_{\beta;B_{R_1}}.
\end{equation}
Now we choose $\varepsilon$ depending only on $d$ and $\beta$ such that $2\varepsilon^\beta\le 1/2$.
To estimate $I_3$, we consider two cases.
If $h>R$, by the triangle inequality, we have
\begin{align*}
I_3 &\le Ch^{-1-\beta}\big(|u^{(R)}(y+\varepsilon h\ell)+u^{(R)}(y')-2u^{(R)}(\bar y)|\\
&\quad+|u^{(R)}(y'+\varepsilon h\ell)+u^{(R)}(y)-2u^{(R)}(\bar y)|\big),
\end{align*}
where $\bar y=(y+\varepsilon h\ell+y')/2$.
Then by the Minkowski inequality,
\begin{equation}
                                        \label{eq3.03}
I_3\le Ch^{\alpha-1-\beta}[u^{(R)}]_{\Lambda^\alpha(B_{R_1})}\le CR^{\alpha-1-\beta}[u]_{\Lambda^\alpha(B_{2R_1})}.
\end{equation}
On the other hand, if $h\in (0,R)$, by the mean value theorem and \eqref{eq11.15} with $\beta=1$,
\begin{equation}
                                \label{eq3.04}
I_3
\le Ch^{1-\beta}[D u^{(R)}]_{1;B_{R_1}}\le Ch^{1-\beta}R^{\alpha-2}[u]_{\Lambda^\alpha(B_{2R_1})}\le CR^{\alpha-1-\beta}[u]_{\Lambda^\alpha(B_{2R_1})}.
\end{equation}
Combining \eqref{eq3.02}, \eqref{eq3.03}, and \eqref{eq3.04}, we obtain
$$
h^{-\beta}|D_\ell u^{(R)}(x)-D_\ell u^{(R)}(x')|\le \frac 1 2 [Du^{(R)}]_{\beta;B_{R_1}}+CR^{\alpha-1-\beta}[u]_{\Lambda^\alpha(B_{2R_1})}.
$$
Taking the supremum of the left-hand side above with respect to unit vector $\ell\in \bR^d$ and $x,x'\in B_{R_1}$, we immediately get \eqref{eq11.15}. The lemma is proved.
\end{proof}

\begin{lemma}
                                \label{lem2.2}
Let $\alpha\in (0,1)$, $\beta\in (0,1)$, and $R>0$ be constants. Let $p=p(x)$ be the first-order Taylor expansion of $u^{(R)}$ at the origin and $\tilde u=u-p$.
Then for any integer $j\ge 0$, we have
\begin{align}
                                            \label{eq3.51}
&\|\tilde u\|_{L_\infty(B_{2^{j+1} R})}\le C2^{j(1+\beta)}R^\alpha[u]_{\Lambda^\alpha(B_{2^{j+2} R})},\\
&\sup_{\substack{x,x'\in B_{2^jR}\\0<|x-x'|<2R}} \frac{|\tilde u(x)-\tilde u(x')|}{|x-x'|^\alpha}
\le C2^{j\beta}[u]_{\Lambda^\alpha(B_{2^{j+2} R})}, \label{eq 12.62b}
\end{align}
where $C>0$ is a constant depending only on $d$, $\beta$, and $\alpha$.
\end{lemma}
\begin{proof}
Since $\zeta\in C_0^\infty(B_1)$ is radial and has unit integral,
we have for any $x\in B_{2^{j+1} R}$,
\begin{align}
&\big|u^{(R)}(x)-u(x)\big|\nonumber\\
&=\big|\frac 1 2 \int_{\bR^d} \big(u(x+Ry)-u(x-Ry)
-2u(x)\big)\zeta(y)\,dy\Big|\le CR^\alpha[u]_{\Lambda^\alpha(B_{2^{j+2}R})}.
                    \label{eq10.31b}
\end{align}
By the mean value theorem and Lemma \ref{lem2.1}, for any $x\in B_{2^{j+1}R}$,
\begin{equation*}
\big|u^{(R)}(x)-p(x)\big|
\le C(2^{j+1}R)^{1+\beta}[u^{(R)}]_{1+\beta;B_{2^{j+1}R}}
\le C2^{j(1+\beta)}R^\alpha[u]_{\Lambda^\alpha(B_{2^{j+2}R})},
\end{equation*}
which together with \eqref{eq10.31b} implies \eqref{eq3.51}. Next we show \eqref{eq 12.62b}. For any two distinct points $x,x'\in B_{2^j R}$ satisfying $0<|x-x'|<2R$, denote $h=|x-x'|(<2R)$. Let $k$ be the largest nonnegative integer such that $2^k (x'-x)+x\in B_{2^{j+1}R}$. Clearly,
\begin{equation}
                                    \label{eq4.26}
2^k h\in (2^{j-1}R,2^{j+2}R).
\end{equation}
It follows from \eqref{eq10.22} that
\begin{align}
                            \label{eq10.12}
&\tilde u(x')-\tilde u(x)=
2^{-k}
\big(\tilde u(2^k (x'-x)+x)-\tilde u(x)\big)\nonumber\\
&\quad+\sum_{i=0}^{k-1} 2^{-i-1}
\big(2\tilde u(2^i (x'-x)+x)-\tilde u(x)-\tilde u(2^{i+1} (x'-x)+x)\big).
\end{align}
By \eqref{eq4.26}, \eqref{eq10.12}, and \eqref{eq3.51},
we obtain
\begin{align*}
h^{-\alpha}|\tilde u(x')-\tilde u(x)|
&\le 2^{-k+1}h^{-\alpha}\|\tilde u\|_{L_\infty(B_{2^{j+1} R})}+C[u]_{\Lambda^\alpha(B_{2^{j+1} R})}\\
&\le C2^{-j}R^{-1} h^{1-\alpha}\cdot 2^{j(1+\beta)}R^\alpha[u]_{\Lambda^\alpha(B_{2^{j+2} R})}
+C[u]_{\Lambda^\alpha(B_{2^{j+1} R})}\\
&\le C2^{j\beta}[u]_{\Lambda^\alpha(B_{2^{j+2} R})},
\end{align*}
where we used $h<2R$ in the last inequality. The lemma is proved.
\end{proof}

\section{Proofs}

The following proposition is a further refinement of \cite[Corollary 4.6]{DZ16}.
\begin{proposition}\label{prop1.1}
Let $\sigma\in (0,2)$ and $0<\lambda\le \Lambda$. Assume that for any $\beta\in\mathcal{A}$, $K_{\beta}$ only depends on $y$. There is a constant $\hat{\alpha}\in (0,1)$ depending on $d,\sigma,\lambda$, and $\Lambda$ so that the following holds. Let $\alpha\in (0,\hat{\alpha})$. 
Suppose $u\in C^{\sigma+\alpha}(B_1)\cap C^{\alpha}(\bR^d)$ is a solution of
\begin{equation*}
\inf_{\beta\in \mathcal{A}}(L_{\beta} u+f_{\beta})=0\quad \text{in}\,\, B_1.
\end{equation*}
Then,
\begin{equation*}
[u]_{\alpha+\sigma;B_{1/2}}\le 
C\sum_{j=1}^\infty 2^{-j\sigma}M_j
+C\sup_{\beta}[f_{\beta}]_{\alpha; B_{1}},
\end{equation*}
where
$$
M_j=\sup_{x,x'\in B_{2^j},0<|x-x'|<2}\frac{|u(x)-u(x')|}{|x-x'|^\alpha}.
$$
\end{proposition}
\begin{proof}
This follows from the proof of \cite[Corollary 4.6]{DZ16}
by observing that in the estimate of $[h_{\beta}]_{\alpha;B_1}$, the term
$[u]_{\alpha; B_{2^{j}}}$ can be replaced by $M_j$. Moreover, by replacing $u$ by $u-u(0)$, we see that
$$
\|u\|_{\alpha;B_2}\le C[u]_{\alpha;B_2}.
$$
The lemma is proved.
\end{proof}

\begin{proposition}\label{prop3.2}
Suppose that \eqref{eq 1} is satisfied in $\bR^d$. Then under the conditions of Theorem \ref{thm 1}, we have
\begin{equation}
                                        \label{eq8.16}
[u]_\sigma\le C\|u\|_{L_\infty}+C
\sum_{k=1}^\infty \omega_f(2^{-k}),
\end{equation}
where $C>0$ is a constant depending only on $d$, $\lambda$, $\Lambda$, $\omega_a$, and $\sigma$.
\end{proposition}
\begin{proof}
{\bf Case 1: $\sigma\in (0,1)$.} For $k\in \bN$, let $v$ be the solution of
\begin{align}\label{eq 12.181}
\begin{cases}
\inf_{\beta\in\mathcal{A}}\big(L_{\beta}(0)v+f_{\beta}(0)\big)=0\quad &\text{in}\,\, B_{2^{-k}}\\
v=u\quad &\text{in}\,\,B_{2^{-k}}^c
\end{cases},
\end{align}
where $L_{\beta}(0)$ is the operator with kernel $K_{\beta}(0,y)$. Then by Proposition \ref{prop1.1} with scaling, we have
\begin{align}
&[v]_{\alpha+\sigma;B_{2^{-k-1}}}\le 
C\sum_{j=1}^\infty 2^{(k-j)\sigma}M_j+C2^{k\sigma}[v]_{\alpha;B_{2^{-k}}}
\nonumber\\
&\le C\sum_{j=1}^k 2^{(k-j)\sigma}M_j+C[u]_{\alpha}+C2^{k\sigma}[v]_{\alpha;B_{2^{-k}}},
                        \label{eq1.04}
\end{align}
where $\alpha\in (0,\hat\alpha)$ satisfying $\sigma+\alpha<1$ and
$$
M_j=\sup_{x,x'\in B_{2^{j-k}},0<|x-x'|<2^{-k+1}}\frac{|u(x)-u(x')|}{|x-x'|^\alpha}.
$$
Let $k_0,k_1\ge 1$ be integers to be specified. From \eqref{eq1.04}, we get
\begin{align}
[v]_{\alpha;B_{2^{-k-k_0}}}\le
C2^{-(k+k_0)\sigma}\sum_{j=1}^k 2^{(k-j)\sigma}M_j+C2^{-(k+k_0)\sigma}[u]_{\alpha}
+C2^{-k_0\sigma}[v]_{\alpha;B_{2^{-k}}}.
                        \label{eq1.04b}
\end{align}

Next, $w:=u-v$ satisfies
\begin{equation}
                                            \label{eq4.25}
\begin{cases}
 \cM^+w\ge -C_k\quad &\text{in}\,\, B_{2^{-k}},\\
 \cM^-w\le C_k\quad &\text{in}\,\, B_{2^{-k}},\\
 w=0\quad &\text{in}\,\, B_{2^{-k}}^c,
\end{cases}
\end{equation}
where
$$
C_k=\sup_{\beta\in \cA}\|f_{\beta}-f_{\beta}(0)+(L_{\beta}-L_{\beta}(0))u\|_{L_\infty(B_{2^{-k}})}.
$$
It is easily seen that
\begin{align*}
C_k&\le \omega_f(2^{-k})+C\omega_a(2^{-k})\int_{\bR^d}|u(x+y)-u(x)||y|^{-d-\sigma}\,dy\\
&\le \omega_f(2^{-k})+C\omega_a(2^{-k})\Big(\sup_{x_0\in B_{2^{-k}}}\sum_{j=0}^\infty 2^{j(\sigma-\alpha)} [u]_{\alpha;B_{2^{-j}}(x_0)}+\|u\|_{L_\infty}\Big).
\end{align*}
Then by the H\"older estimate \cite[Lemma 2.5]{DZ16}, we have
\begin{align}
                                \label{eq3.05}
&[w]_{\alpha;B_{2^{-k}}}\le C2^{-k(\sigma-\alpha)}C_k\nonumber\\
&\le C2^{-k(\sigma-\alpha)}\Big(\omega_f(2^{-k})
+\omega_a(2^{-k})\big(\sup_{x_0\in B_{2^{-k}}}\sum_{j=0}^\infty 2^{j(\sigma-\alpha)} [u]_{\alpha;B_{2^{-j}}(x_0)}+\|u\|_{L_\infty}\big)\Big).
\end{align}
Combining \eqref{eq1.04b} and \eqref{eq3.05} yields
\begin{align}
                            \label{eq3.25}
&2^{(k+k_0)(\sigma-\alpha)}[u]_{\alpha;B_{2^{-k-k_0}}}\nonumber\\
&\le
C2^{-(k+k_0)\alpha}\sum_{j=1}^k 2^{(k-j)\sigma}[u]_{\alpha;B_{2^{j-k}}}+C2^{-(k+k_0)\alpha}[u]_{\alpha}
+C2^{-k_0\alpha}2^{k(\sigma-\alpha)}[u]_{\alpha;B_{2^{-k}}}\nonumber\\
&\,\,+C2^{k_0(\sigma-\alpha)}\Big(\omega_f(2^{-k})
+\omega_a(2^{-k})\big(\sup_{x_0\in B_{2^{-k}}}\sum_{j=0}^\infty 2^{j(\sigma-\alpha)} [u]_{\alpha;B_{2^{-j}}(x_0)}+\|u\|_{L_\infty}\big)\Big).
\end{align}
Shifting the coordinates, from \eqref{eq3.25} we get
\begin{align}
                            \label{eq3.25b}
&2^{(k+k_0)(\sigma-\alpha)}\sup_{x_0\in \bR^d}[u]_{\alpha;B_{2^{-k-k_0}}(x_0)}\nonumber\\
&\le
C2^{-(k+k_0)\alpha}\sup_{x_0\in \bR^d}\sum_{j=1}^k 2^{(k-j)\sigma}[u]_{\alpha;B_{2^{j-k}}(x_0)}
+C2^{-(k+k_0)\alpha}[u]_{\alpha}\nonumber\\
&\quad +C2^{-k_0\alpha}2^{k(\sigma-\alpha)}\sup_{x_0\in \bR^d}[u]_{\alpha;B_{2^{-k}(x_0)}}
+C2^{k_0(\sigma-\alpha)}\Big(\omega_f(2^{-k})\nonumber\\
&\quad
+\omega_a(2^{-k})(\sup_{x_0\in \bR^d}\sum_{j=0}^\infty 2^{j(\sigma-\alpha)} [u]_{\alpha;B_{2^{-j}}(x_0)}+\|u\|_{L_\infty})\Big).
\end{align}
We take the summation of \eqref{eq3.25b} in $k=k_1,k_1+1,\ldots$ to obtain
\begin{align*}
&\sum_{k=k_1}^\infty 2^{(k+k_0)(\sigma-\alpha)}
\sup_{x_0\in \bR^d}[u]_{\alpha;B_{2^{-k-k_0}}(x_0)}\\
&\le
C\sum_{k=k_1}^\infty 2^{-(k+k_0)\alpha}\Big(\sup_{x_0\in \bR^d}\sum_{j=1}^k 2^{(k-j)\sigma}[u]_{\alpha;B_{2^{j-k}}(x_0)}\Big)
+C2^{-(k_1+k_0)\alpha}[u]_{\alpha}\\
&\quad+C2^{-k_0\alpha}\sum_{k=k_1}^\infty2^{k(\sigma-\alpha)}\sup_{x_0\in \bR^d}[u]_{\alpha;B_{2^{-k}(x_0)}}+C2^{k_0(\sigma-\alpha)}
\sum_{k=k_1}^\infty\Big(\omega_f(2^{-k})
\\
&\quad+\omega_a(2^{-k})\big(\sum_{j=0}^\infty 2^{j(\sigma-\alpha)} \sup_{x_0\in \bR^d}[u]_{\alpha;B_{2^{-j}}(x_0)}+\|u\|_{L_\infty}\big)\Big),
\end{align*}
which by switching the order of summations is further bounded by
\begin{align*}
&C2^{-k_0\alpha}\sum_{j=0}^\infty 2^{j(\sigma-\alpha)}\sup_{x_0\in \bR^d}[u]_{\alpha;B_{2^{-j}}(x_0)}\\
&\quad +C2^{-(k_1+k_0)\alpha}[u]_{\alpha}+C2^{k_0(\sigma-\alpha)}
\sum_{k=k_1}^\infty \omega_f(2^{-k})\\
&\quad+C2^{k_0(\sigma-\alpha)}\sum_{k=k_1}^\infty\omega_a(2^{-k})
\cdot\Big(\sum_{j=0}^\infty 2^{j(\sigma-\alpha)}\sup_{x_0\in \bR^d} [u]_{\alpha;B_{2^{-j}}(x_0)}+\|u\|_{L_\infty}\Big).
\end{align*}
The bound above together with the obvious inequality
$$
\sum_{j=0}^{k_1+k_0-1}2^{j(\sigma-\alpha)}\sup_{x_0\in \bR^d}[u]_{\alpha;B_{2^{-j}}(x_0)}
\le C2^{(k_1+k_0)(\sigma-\alpha)}[u]_\alpha,
$$
implies that
\begin{align*}
&\sum_{j=0}^\infty 2^{j(\sigma-\alpha)}\sup_{x_0\in \bR^d}[u]_{\alpha;B_{2^{-j}}(x_0)}\le
C2^{-k_0\alpha}\sum_{j=0}^\infty 2^{j(\sigma-\alpha)}\sup_{x_0\in \bR^d}[u]_{\alpha;B_{2^{-j}}(x_0)}
\\
&\quad +C2^{(k_1+k_0)(\sigma-\alpha)}[u]_{\alpha}+C2^{k_0(\sigma-\alpha)}
\sum_{k=k_1}^\infty \omega_f(2^{-k})\\
&\quad +C2^{k_0(\sigma-\alpha)}\sum_{k=k_1}^\infty\omega_a(2^{-k})\cdot
\Big(\sum_{j=0}^\infty 2^{j(\sigma-\alpha)} \sup_{x_0\in \bR^d}[u]_{\alpha;B_{2^{-j}}(x_0)}
+C\|u\|_{L_\infty}\Big).
\end{align*}
By first choosing $k_0$ sufficiently large and then $k_1$ sufficiently large, we get
$$
\sum_{j=0}^\infty 2^{j(\sigma-\alpha)}\sup_{x_0\in \bR^d}[u]_{\alpha;B_{2^{-j}}(x_0)}\le
C\|u\|_{\alpha}+C
\sum_{k=1}^\infty \omega_f(2^{-k}),
$$
which together with Lemma \ref{lem2.3} (i) and the interpolation inequality gives \eqref{eq8.16}.

{\bf Case 2: $\sigma\in (1,2)$.} For $k\in \bN$, let $v_M$ be the solution of
\begin{align*}
\begin{cases}
\inf_{\beta\in\mathcal{A}}\big(L_{\beta}(0)v_M+f_{\beta}(0)\big)=0\quad &\text{in}\,\, B_{2^{-k}}\\
v_M=g_M\quad &\text{in}\,\,B_{2^{-k}}^c
\end{cases},
\end{align*}
where $M\ge 2\|u-p_0\|_{L_\infty(B_{2^{-k}})}$ is a constant to be specified later,
\begin{align*}
g_M = \max\big(\min(u-p_0,M),-M\big),
\end{align*}
and  $p_0$ is the first-order Taylor's expansion of $u$ at the origin.

By Proposition \ref{prop1.1}, instead of \eqref{eq1.04}, we have
\begin{align}
&[v_M]_{\alpha+\sigma;B_{2^{-k-1}}}\le 
C\sum_{j=0}^\infty 2^{(k-j)\sigma}M_j+C2^{k\sigma}[v_M]_{\alpha;B_{2^{-k}}}
\nonumber\\
&\le C\sum_{j=0}^k 2^{(k-j)\sigma}M_j+C\|Du\|_{L_\infty}+C2^{k\sigma}[v_M]_{\alpha;B_{2^{-k}}},
                        \label{eq4.22}
\end{align}
where $\alpha\in (0,\hat\alpha)$ and
$$
M_j=\sup_{x,x'\in B_{2^{j-k}},0<|x-x'|<2^{-k+1}}\frac{|u(x)-p_0(x)-u(x')+p_0(x')|}{|x-x'|^\alpha}.
$$
From \eqref{eq4.22} and the mean value formula,
\begin{align*}
&\|v_M-p_1\|_{L_\infty(B_{2^{-k-k_0}})}\le
C2^{-(k+k_0)(\sigma+\alpha)}\sum_{j=0}^k 2^{(k-j)\sigma}M_j\\
&\quad+C2^{-(k+k_0)(\sigma+\alpha)}\|Du\|_{L_\infty}
+C2^{-k\alpha-k_0(\sigma+\alpha)}[v_M]_{\alpha;B_{2^{-k}}},
\end{align*}
where $p_1$ is the first-order Taylor's expansion of $v_M$ at the origin.
The above inequality, \eqref{eq4.22}, and the interpolation inequality imply
\begin{align}
&[v_M-p_1]_{\alpha;B_{2^{-k-k_0}}}\le
C2^{-(k+k_0)\sigma}\sum_{j=0}^k 2^{(k-j)\sigma}M_j\nonumber\\
&\quad+C2^{-(k+k_0)\sigma}\|Du\|_{L_\infty}+ C2^{-k_0\sigma}[v_M]_{\alpha;B_{2^{-k}}},
                        \label{eq1.04bb}
\end{align}
Next $w_M:=g_M-v_M$ satisfies
\begin{equation*}
\begin{cases}
 \cM^+w_M\ge h_M-C_k\quad &\text{in}\,\, B_{2^{-k}},\\
 \cM^-w_M\le \hat{h}_M+C_k\quad &\text{in}\,\, B_{2^{-k}},\\
 w_M=0\quad &\text{in}\,\, B_{2^{-k}}^c,
\end{cases}
\end{equation*}
where
\begin{equation*}
h_M :=\cM^-(g_M-(u-p_0)),\quad \hat{h}_M:=\cM^+(g_M-(u-p_0)).
\end{equation*}
By the dominated convergence theorem, it is easy to see that
\begin{equation*}
\|h_M\|_{L_\infty(B_{2^{-k}})},\,\,
\|\hat{h}_M\|_{L_\infty(B_{2^{-k}})}\rightarrow 0\quad \text{as}\quad M\rightarrow \infty.
\end{equation*}
By the same argument as in the previous case,
$$
C_k\le \omega_f(2^{-k})+C\omega_a(2^{-k})\Big(\sup_{x_0\in \bR^d}\sum_{j=0}^\infty 2^{j(\sigma-\alpha)} [u-P_{x_0}u]_{\alpha;B_{2^{-j}}(x_0)}+\|Du\|_{L_\infty}\Big).
$$
Thus similar to \eqref{eq3.05}, choosing $M$ sufficiently large so that
\begin{equation*}
 \|h_M\|_{L_\infty(B_{2^{-k}})},\,\, \|\hat{h}_M\|_{L_\infty(B_{2^{-k}})}\le C_k/2,
\end{equation*}
we have
\begin{align}
                                \label{eq3.05bb}
[w_M]_{\alpha;B_{2^{-k}}}
&\le C2^{-k(\sigma-\alpha)}\Big(\omega_f(2^{-k})
+\omega_a(2^{-k})\|Du\|_{L_\infty}\nonumber\\
&\quad+\omega_a(2^{-k})\sup_{x_0\in \bR^d}\sum_{j=0}^\infty 2^{j(\sigma-\alpha)} [u-P_{x_0}u]_{\alpha;B_{2^{-j}}(x_0)}\Big).
\end{align}
Combining \eqref{eq1.04bb} and \eqref{eq3.05bb}, similar to \eqref{eq3.25b}, we obtain
\begin{align}
                            \label{eq3.25bb}
&2^{(k+k_0)(\sigma-\alpha)}\sup_{x_0\in \bR^d}\inf_{p\in \cP_1}[u-p]_{\alpha;B_{2^{-k-k_0}}(x_0)}\nonumber\\
&\le
C2^{-(k+k_0)\alpha}\sup_{x_0\in \bR^d}\sum_{j=0}^k 2^{(k-j)\sigma}[u-P_{x_0}u]_{\alpha;B_{2^{j-k}}(x_0)}
+C2^{-(k+k_0)\alpha}\|Du\|_{L_\infty} \nonumber\\
&\quad +C2^{-k_0\alpha}2^{k(\sigma-\alpha)}\sup_{x_0\in \bR^d}[u-P_{x_0}]_{\alpha;B_{2^{-k}(x_0)}}+C2^{k_0(\sigma-\alpha)}
\Big(\omega_f(2^{-k})
\nonumber\\
&\quad +\omega_a(2^{-k})(\sup_{x_0\in \bR^d}\sum_{j=0}^\infty 2^{j(\sigma-\alpha)} [u-P_{x_0}u]_{\alpha;B_{2^{-j}}(x_0)}+\|Du\|_{L_\infty})\Big).
\end{align}
Using \eqref{eq3.25bb}, as before we get
\begin{align}
                                \label{eq9.06}
&\sum_{k=k_1}^\infty 2^{(k+k_0)(\sigma-\alpha)}\sup_{x_0\in \bR^d}\inf_{p\in \cP_1}[u-p]_{\alpha;B_{2^{-k-k_0}}(x_0)}\nonumber\\
&\le
C2^{-k_0\alpha}\sum_{j=0}^\infty 2^{j(\sigma-\alpha)}\sup_{x_0\in \bR^d}[u-P_{x_0}u]_{\alpha;B_{2^{-j}}(x_0)}\nonumber\\
&\quad+C2^{-(k_1+k_0)\alpha}\|u\|_{1}+C2^{k_0(\sigma-\alpha)}
\sum_{k=k_1}^\infty \omega_f(2^{-k})\nonumber\\
&\quad+C2^{k_0(\sigma-\alpha)}\sum_{k=k_1}^\infty\omega_a(2^{-k})\cdot\sup_{x_0\in \bR^d}\sum_{j=0}^\infty 2^{j(\sigma-\alpha)} [u-P_{x_0}u]_{\alpha;B_{2^{-j}}(x_0)},
\end{align}
and
\begin{align*}
&\sum_{j=0}^\infty 2^{j(\sigma-\alpha)}\sup_{x_0\in \bR^d}\inf_{p\in \cP_1}[u-p]_{\alpha;B_{2^{-j}}(x_0)}\\
&\le
C2^{-k_0\alpha}\sum_{j=0}^\infty 2^{j(\sigma-\alpha)}\sup_{x_0\in \bR^d}[u-P_{x_0}u]_{\alpha;B_{2^{-j}}(x_0)}\\
&\quad+C2^{(k_1+k_0)(\sigma-\alpha)}\|u\|_{1}+C2^{k_0(\sigma-\alpha)}
\sum_{k=k_1}^\infty \omega_f(2^{-k})\\
&\quad+C2^{k_0(\sigma-\alpha)}\sum_{k=k_1}^\infty\omega_a(2^{-k})\cdot\sup_{x_0\in \bR^d}\sum_{j=0}^\infty 2^{j(\sigma-\alpha)} [u-P_{x_0}u]_{\alpha;B_{2^{-j}}(x_0)}.
\end{align*}
By choosing $k_0$ and $k_1$ sufficiently large and applying Lemma \ref{lem2.4}, we obtain
\begin{equation}
                                        \label{eq12.03}
\sum_{j=0}^\infty 2^{j(\sigma-\alpha)}\sup_{x_0\in \bR^d}\inf_{p\in \cP_1}[u-p]_{\alpha;B_{2^{-j}}(x_0)}\le
C\|u\|_{1}+C\sum_{k=1}^\infty \omega_f(2^{-k}).
\end{equation}
Finally, by Lemma \ref{lem2.3} (ii) and the interpolation inequality,
we get \eqref{eq8.16}.

{\bf Case 3: $\sigma=1$.} We proceed as in the previous case, but instead take $p_0$ to be the first-order Taylor's expansion of the mollification $u^{(2^{-k})}$ at the origin. We also assume that the solution $v$ to \eqref{eq 12.181} exists without carrying out another approximation argument.
By Proposition \ref{prop1.1} and Lemma \ref{lem2.2} with $\beta=\alpha/2$,
\begin{align}
&[v]_{\alpha+1;B_{2^{-k-1}}}\le 
C\sum_{j=0}^\infty 2^{k-j}M_j+C2^k[v]_{\alpha;B_{2^{-k}}}\nonumber\\
&\le C\sum_{j=0}^\infty 2^{k-j+j\alpha/2}[u]_{\Lambda^\alpha(B_{2^{j-k}})}
+C2^k[v]_{\alpha;B_{2^{-k}}}\nonumber\\
&\le C\sum_{j=0}^k  2^{k-j+j\alpha/2}[u]_{\Lambda^\alpha(B_{2^{j-k}})}
+C2^{k\alpha/2}[u]_{\alpha}+C2^k[v]_{\alpha;B_{2^{-k}}}.
                        \label{eq4.22c}
\end{align}
From \eqref{eq4.22c} and the interpolation inequality, we obtain
\begin{align}
&[v-p_1]_{\alpha;B_{2^{-k-k_0}}}\nonumber\\
&\le
C2^{-(k+k_0)}\sum_{j=0}^k 2^{k-j+j\alpha/2}[u]_{\Lambda^\alpha(B_{2^{j-k}})}
+C2^{-(k+k_0)+k\alpha/2}[u]_{\alpha}+C2^{-k_0}[v]_{\alpha;B_{2^{-k}}}\nonumber\\
&\le
C2^{-(k+k_0)}\sum_{j=0}^k 2^{k-j+j\alpha/2}\inf_{p\in \cP_1}[u-p]_{\alpha;B_{2^{j-k}}}\nonumber\\
&\quad+C2^{-(k+k_0)+k\alpha/2}[u]_{\alpha}+C2^{-k_0}[v]_{\alpha;B_{2^{-k}}},
                        \label{eq1.04c}
\end{align}
where $p_1$ is the first-order Taylor's expansion of $v$ at the origin.
Next $w:=u-p_0-v$ satisfies \eqref{eq4.25}, where by the cancellation property \eqref{eq10.58},
$$
C_k\le \omega_f(2^{-k})+C\omega_a(2^{-k})\Big(\sup_{x_0\in \bR^d}\sum_{j=0}^\infty 2^{j(1-\alpha)}\inf_{p\in \cP_1} [u-p]_{\alpha;B_{2^{-j}}(x_0)}+\|u\|_{L_\infty}\Big).
$$
Therefore, similar to \eqref{eq3.05}, we have
\begin{align}
                                \label{eq3.05c}
[w]_{\alpha;B_{2^{-k}}}
&\le C2^{-k(1-\alpha)}\Big(\omega_f(2^{-k})\nonumber\\
&\quad+\omega_a(2^{-k})\big(\sup_{x_0\in \bR^d}\sum_{j=0}^\infty 2^{j(1-\alpha)} \inf_{p\in \cP_1} [u-p]_{\alpha;B_{2^{-j}}(x_0)}+\|u\|_{L_\infty}\big)\Big).
\end{align}
Notice that from \eqref{eq 12.62b} and the triangle inequality
\begin{align*}
&[v]_{\alpha;B_{2^{-k}}}
\le [w]_{\alpha;B_{2^{-k}}}+ [u-p_0]_{\alpha;B_{2^{-k}}}\\
&\le [w]_{\alpha;B_{2^{-k}}}+C[u]_{\Lambda^\alpha(B_{2^{-k+2}})}
\le [w]_{\alpha;B_{2^{-k}}}+C\inf_{p\in \cP_1}[u-p]_{\alpha;B_{2^{-k+2}}}.
\end{align*}

Similar to \eqref{eq3.25b}, combining \eqref{eq1.04c}, \eqref{eq3.05c}, and the inequality above, we obtain
\begin{align*}
&2^{(k+k_0)(1-\alpha)}\sup_{x_0\in \bR^d}\inf_{p\in \cP_1}[u-p]_{\alpha;B_{2^{-k-k_0}}(x_0)}\\
&\le
C2^{-(k+k_0)\alpha}\sup_{x_0\in \bR^d}\sum_{j=0}^k 2^{k-j+j\alpha/2}\inf_{p\in \cP_1}[u-p]_{\alpha;B_{2^{j-k}}(x_0)}+C2^{-(k/2+k_0)\alpha}[u]_{\alpha}\nonumber\\
&\quad + C2^{-k_0\alpha+(k-2)(1-\alpha)}\sup_{x_0\in \bR^d}\inf_{p\in \cP_1}[u-p]_{\alpha;B_{2^{-k+2}(x_0)}}+C2^{k_0(1-\alpha)}\Big(\omega_f(2^{-k})
\\
&\quad +\omega_a(2^{-k})\big(\sup_{x_0\in \bR^d}\sum_{j=0}^\infty 2^{j(1-\alpha)} \inf_{p\in \cP_1}[u-p]_{\alpha;B_{2^{-j}}(x_0)}+\|u\|_{L_\infty}\big)\Big),
\end{align*}
which by summing in $k=k_1,k_1+1,\ldots$ implies that
\begin{align*}
&\sum_{k=k_1}^\infty 2^{(k+k_0)(1-\alpha)}\sup_{x_0\in \bR^d}\inf_{p\in \cP_1}[u-p]_{\alpha;B_{2^{-k-k_0}}(x_0)}\\
&\le
C2^{-k_0\alpha}\sum_{j=0}^\infty 2^{j(1-\alpha)}\sup_{x_0\in \bR^d}\inf_{p\in \cP_1}[u-p]_{\alpha;B_{2^{-j}}(x_0)}\\
&\quad+C2^{-(k/2+k_0)\alpha}[u]_{\alpha}+C2^{k_0(1-\alpha)}
\sum_{k=k_1}^\infty\omega_f(2^{-k})
+C2^{k_0(1-\alpha)}\sum_{k=k_1}^\infty\omega_a(2^{-k})\\
&\qquad\cdot(\sup_{x_0\in \bR^d}\sum_{j=0}^\infty 2^{j(1-\alpha)} \inf_{p\in \cP_1}[u-p]_{\alpha;B_{2^{-j}}(x_0)}+\|u\|_{L_\infty}),
\end{align*}
where for the first term on the right-hand side, we switched the order of summations to get
\begin{align*}
&\sum_{k=k_1}^\infty2^{-(k+k_0)\alpha}\sup_{x_0\in \bR^d}\sum_{j=0}^k 2^{k-j+j\alpha/2}\inf_{p\in \cP_1}[u-p]_{\alpha;B_{2^{j-k}}(x_0)}\\
&\le \sum_{k=0}^\infty2^{-(k+k_0)\alpha}\sum_{j=0}^k 2^{j+(k-j)\alpha/2}
\sup_{x_0\in \bR^d}\inf_{p\in \cP_1}[u-p]_{\alpha;B_{2^{-j}}(x_0)}\\
&=2^{-k_0\alpha}\sum_{j=0}^\infty 2^{j(1-\alpha/2)}\sup_{x_0\in \bR^d}\inf_{p\in \cP_1}[u-p]_{\alpha;B_{2^{-j}}(x_0)}\sum_{k=j}^\infty 2^{-k\alpha/2}\\
&\le C2^{-k_0\alpha}\sum_{j=0}^\infty 2^{j(1-\alpha)}\sup_{x_0\in \bR^d}\inf_{p\in \cP_1}[u-p]_{\alpha;B_{2^{-j}}(x_0)}.
\end{align*}
Therefore,
\begin{align*}
&\sum_{j=0}^\infty 2^{j(1-\alpha)}\sup_{x_0\in \bR^d}\inf_{p\in \cP_1}[u-p]_{\alpha;B_{2^{-j}}(x_0)}\\
&\le
C2^{-k_0\alpha}\sum_{j=0}^\infty 2^{j(1-\alpha)}\sup_{x_0\in \bR^d}\inf_{p\in \cP_1}[u-p]_{\alpha;B_{2^{-j}}(x_0)}
\\
&\quad+C2^{(k_1+k_0)(1-\alpha)}[u]_{\alpha}+C2^{k_0(1-\alpha)}
\sum_{k=k_1}^\infty\omega_f(2^{-k})
+C2^{k_0(1-\alpha)}\sum_{k=k_1}^\infty\omega_a(2^{-k})\\
&\qquad\cdot(\sum_{j=0}^\infty 2^{j(1-\alpha)}\sup_{x_0\in \bR^d} \inf_{p\in \cP_1}[u-p]_{\alpha;B_{2^{-j}}(x_0)}+\|u\|_{L_\infty}),
\end{align*}
Finally, to get \eqref{eq8.16} it suffices to choose $k_0$ and $k_1$ sufficiently large and apply Lemma \ref{lem2.3} (iii).
\end{proof}

Next we employ a localization argument as in \cite{DZ16}.
\begin{proof}[Proof of Theorem \ref{thm 1}]
Since the proof of the case when $\sigma\in(0,1)$ is almost the same as $\sigma\in(1,2)$ and actually simpler, we only present the latter and sketch the proof of the case when $\sigma = 1$ in the end.

{\bf The case when $\sigma\in(1,2)$.} We divide the proof into three steps.

{\em Step 1.} For $k=1,2,\ldots$, denote $B^k := B_{1-2^{-k}}$.
Let $\eta_k\in C_0^\infty(B^{k+1})$ be a sequence of nonnegative smooth cutoff functions satisfying
$\eta\equiv 1$ in $B^{k}$, $|\eta|\le 1$ in $B^{k+1}$, and $\|D^i\eta_k\|_{L_\infty}\le C2^{ki}$ for each $i\ge 0$. Set $v_k := u\eta_k\in C^{\sigma+}$.  A simple calculation reveals that
\begin{equation*}
\inf_{\beta\in \mathcal{A}}(L_\beta v_k-h_{k\beta}+\eta_k f_\beta)=0\quad \text{in}\,\, \bR^d,
\end{equation*}
where
\begin{equation*}
h_{k\beta}=h_{k\beta}(x) = \int_{\bR^d}\frac{\xi_k(x,y)a_\beta(x,y)}{|y|^{d+\sigma}}\,dy
\end{equation*}
and
\begin{equation*}
\xi_k(x,y) = u(x+y)(\eta_k(x+y)-\eta_k(x))-y\cdot D\eta_k(x)u(x).
\end{equation*}
Obviously, $\eta_k f_\beta$ is a Dini continuous function in $\bR^d$ and
\begin{align*}
&|\eta_k(x)f_\beta(x)-\eta_k(x')f_\beta(x')|\\
&\le \|\eta_k\|_{L_\infty}\omega_f(|x-x'|)+\|f_\beta\|_{L_\infty(B_1)}\|D\eta_k\|_{L_\infty}|x-x'|\\
&\le \omega_f(|x-x'|)+C2^{k}\|f_\beta\|_{L_\infty(B_1)}|x-x'|,
\end{align*}
where $C$ only depends on $d$.

{\em Step 2.} We first estimate the $L_\infty$ norm of $h_{k\beta}$.
By the fundamental theorem of calculus,
\begin{align*}
\xi_k(x,y) = y\cdot\int_{0}^1 u(x+y)D\eta_k(x+ty)-u(x)D\eta_k(x)\,dt.
\end{align*}
For $|y|\ge 2^{-k-3}$, $|\xi_k(x,y)|\le C2^{k}|y|\|u\|_{L_\infty}$. For $|y|<2^{-k-3}$, we can further write
\begin{equation*}
\xi_k(x,y) = y\cdot\int_{0}^1(u(x+y)-u(x))D\eta_k(x+ty)+u(x)(D\eta_k(x+ty)-D\eta_k(x))\,dt,
\end{equation*}
where the second term on the right-hand side is bounded by $C2^{2k}|y|^2|u(x)|$. To estimate the first term, we consider two cases: when $|x|\ge1-2^{-k-2}$, because $|y|<2^{-k-3}$, $\xi_k(x,y)\equiv 0$; when $|x|<1-2^{-k-2}$, we have
\begin{equation*}
\Big|y\cdot\int_0^1(u(x+y)-u(x))D\eta_k(x+ty)\,dt\Big|\le C2^{k}|y|^2\|Du\|_{L_\infty(B^{k+3})}.
\end{equation*}
Hence for $|y|<2^{-k-3}$,
\begin{equation*}
|\xi_k(x,y)|\le C|y|^2\big(2^{2k}|u(x)|+2^{k}\|Du\|_{L_\infty(B^{k+3})}\big).
\end{equation*}
Combining with the case when $|y|>2^{-k-3}$, we see that
\begin{equation}
                                \label{eq11.26}
\|h_{k\beta}\|_{L_\infty}\le C2^{\sigma k}\big(\|u\|_{L_\infty}+\|Du\|_{L_\infty(B^{k+3})}\big).
\end{equation}

Next we estimate the modulus of continuity of $h_{k\beta}$.
By the triangle inequality,
\begin{align}
                                    \label{eq11.39}
&|h_{k\beta}(x)-h_{k\beta}(x')|\nonumber \\
&\le \int_{\bR^d}\frac{|(\xi_k(x,y)-\xi_k(x',y))a_\beta(x,y)|}{|y|^{d+\sigma}}
+\frac{|\xi_k(x',y)(a_\beta(x,y)-a_\beta(x',y))|}{|y|^{d+\sigma}}\,dy\nonumber\\
&:= {\rm I}+{\rm II}.
\end{align}
Similar to \eqref{eq11.26}, by the estimates of $|\xi_k(x,y)|$ above, we have
\begin{equation}
                                \label{eq11.40}
{\rm II}\le C2^{\sigma k}\big(\|u\|_{L_\infty}+\|Du\|_{L_\infty(B^{k+3})}\big)\omega_a(|x-x'|),
\end{equation}
where $C$ depends on $d$, $\sigma$, and  $\Lambda$. For $I$, by the fundamental theorem of calculus,
\begin{align*}
&\xi_k(x,y)-\xi_k(x',y) = y\cdot\int_0^1\Big(u(x+y)D\eta_k(x+ty)-u(x)D\eta_k(x)\\
&\quad -u(x'+y)D\eta_k(x',x+ty)+u(x')D\eta_k(x')\Big)\,dt.
\end{align*}
When $|y|\ge2^{-k-3}$, similar to the estimate of $\xi_k(x,y)$, it follows that
\begin{equation}
                                    \label{eq11.41}
|\xi_k(x,y)-\xi_k(x',y)|\le C|y|\big(2^{k}\omega_u(|x-x'|)+2^{2k}\|u\|_{L_\infty}|x-x'|\big).
\end{equation}
The case when $|y|<2^{-k-3}$ is a bit more delicate. First,  by the fundamental theorem of calculus,
\begin{align*}
&|\xi_k(x,y)-\xi_k(x',y)|\\
&\le |y|\int_0^1|(u(x+y)-u(x))D_k\eta(x+ty)-(u(x'+y)-u(x'))D\eta_k(x'+ty)|\,dt\\
&\quad +|y|^2\int_0^1\int_0^1|u(x)D^2\eta_k(x+tsy)-u(x')D^2\eta_k(x'+tsy)|\,dt\,ds
:= {\rm III}+{\rm IV}.
\end{align*}
It is easily seen that
\begin{align*}
{\rm IV}\le C|y|^2(2^{2k}\omega_u(|x-x'|)+2^{3k}\|u\|_{L^\infty}|x-x'|).
\end{align*}
Next we bound ${\rm III}$ by considering four cases. When $x,x' \in (B^{k+2})^c$, we have ${\rm III}\equiv 0$. When $x,x'\in B^{k+2}$,
\begin{align*}
{\rm III} &\le |y|^2 \int_0^1\int_0^1 |Du(x+sy)D\eta_k(x+ty)-Du(x'+sy)D\eta_k(x'+ty)|\,ds\,dt\\
&\le C|y|^2\big(2^{k}[u]_{1+\alpha;B^{k+3}}|x-x'|^\alpha
+2^{2k}\|Du\|_{L_\infty(B^{k+3})}|x-x'|\big),
\end{align*}
where we choose $\alpha = \frac{\sigma-1}{2}.$
When $x\in B^{k+2}$ and $x'\in (B^{k+2})^c$,
\begin{align*}
{\rm III}& = |y|\int_0^1|(u(x+y)-u(x))D\eta_k(x+ty)|\,dt  \\
&\le |y|^2\int_0^1\int_0^1|Du(x+sy)(D\eta_k(x+ty)-D\eta_k(x'+ty))|\,ds\,dt\\
&\le C|y|^22^{2k}\|Du\|_{L_\infty(B^{k+3})}|x-x'|.
\end{align*}
The last case is similar.
In conclusion, we obtain
\begin{align*}
{\rm III}\le C|y|^2\big(2^{k}[u]_{1+\alpha;B^{k+3}}|x-x'|^\alpha
+2^{2k}\|Du\|_{L_\infty(B^{k+3})}|x-x'|\big).
\end{align*}
Combining the estimates of ${\rm III}, {\rm IV}$, and \eqref{eq11.41}, we obtain
\begin{align}
                            \label{eq11.42}
{\rm I} &\le C2^{k(\sigma+1)}\big(\omega_u(|x-x'|)+[u]_{1+\alpha;B^{k+3}}|x-x'|^\alpha\nonumber\\
&\quad +(\|Du\|_{L_\infty(B^{k+3})}+\|u\|_{L_\infty})|x-x'|\big).
\end{align}
By combining \eqref{eq11.39}, \eqref{eq11.40}, and \eqref{eq11.42}, we obtain
\begin{align*}
|h_{k\beta}(x)-h_{k\beta}(x')|&\le \omega_h(|x-x'|),
\end{align*}
where
\begin{align}
                        \label{eq9.32}
&\omega_h(r) := C2^{\sigma k}\big(\|u\|_{L_\infty}+\|Du\|_{L_\infty(B^{k+3})}\big)\omega_a(r)\nonumber\\
&\quad +C2^{k(\sigma+1)}\big(\omega_u(r)+[u]_{1+\alpha;B^{k+3}}r^\alpha
+(\|Du\|_{L_\infty(B^{k+3})}+\|u\|_{L_\infty})r\big)
\end{align}
is a Dini function.

{\em Step 3.} We apply Proposition \ref{prop3.2} to $v_k$ to obtain
\begin{align*}
[v_k]_\sigma&\le C\|v_k\|_{L_\infty}+C\sum_{j=1}^\infty \big(\omega_h(2^{-j})+\omega_f(2^{-j})\big)+C2^k
\sup_\beta\|f_\beta\|_{L_\infty(B_1)}\\
&\le  C\|v_k\|_{L_\infty}+C2^{k(\sigma+1)} \big([u]_{1+\alpha;B^{k+3}}+\|Du\|_{L_\infty(B^{k+3})}+\|u\|_{L_\infty}\big)\\
&\quad +C\sum_{j=1}^\infty\big(2^{k(\sigma+1)}\omega_u(2^{-j})
+\omega_f(2^{-j})\big)+C2^k\sup_\beta\|f_\beta\|_{L_\infty(B_1)},
\end{align*}
where $C$ depends on $d$, $\lambda$, $\Lambda$, $\sigma$, and $\omega_a$, but independent of $k$. Since $\eta_k\equiv 1$ in $B^k$,  it follows that
\begin{align}
		\nonumber
[u]_{\sigma;B^{k}}&\le C2^{k(\sigma+1)}\|u\|_{L_\infty}+C2^{k(\sigma+1)}
\big([u]_{1+\alpha;B^{k+3}}+\|Du\|_{L_\infty(B^{k+3})}\big)\\
		\label{eq12.141}
&\quad +C_0\sum_{j=1}^\infty\big(2^{k(\sigma+1)}\omega_u(2^{-j})
+\omega_f(2^{-j})\big)
+C2^k\sup_\beta\|f_\beta\|_{L_\infty(B_1)}.
\end{align}
By the interpolation inequality, for any $\epsilon\in(0,1)$
\begin{equation}
		\label{eq12.142}
[u]_{1+\alpha;B^{k+3}}+\|Du\|_{L_\infty(B^{k+3})} \le \epsilon[u]_{\sigma;B^{k+3}}+C\epsilon^{-\frac{1+\alpha}{\sigma-(1+\alpha)}}\|u\|_{L_\infty}.
\end{equation}
Recall that $\alpha = \frac{\sigma-1}{2}$ and denote
$$
N := \frac{1+\alpha}{\sigma-(1+\alpha)} = \frac{\sigma+1}{\sigma-1}(>3).
$$
Combining \eqref{eq12.141} and \eqref{eq12.142} with $\epsilon = C_0^{-1}2^{-3k-12N-1}$, we obtain
\begin{align*}
&[u]_{\sigma;B^k}\le C2^{3k+(3k+12N)N}\|u\|_{L_\infty}+ 2^{-12N-1}[u]_{\sigma;B^{k+3}}\\
&\quad +C2^k\sup_\beta\|f_\beta\|_{L_\infty(B_1)}+C
\sum_{j=1}^\infty\big(2^{3k}\omega_u(2^{-j})+\omega_f(2^{-j})\big).
\end{align*}
Then we multiply $2^{-4kN}$ to both sides of the inequality above and get
\begin{align*}
&2^{-4kN}[u]_{\sigma;B^k}\le C2^{3k-kN}\|u\|_{L_\infty}+2^{-4N(k+3)-1}[u]_{\sigma;B^{k+3}}\\
&\quad +C2^{-4kN+k}\sup_\beta\|f_\beta\|_{L_\infty(B_1)}+C2^{-kN}
\sum_{j=1}^\infty\big(\omega_u(2^{-j})+\omega_f(2^{-j})\big).
\end{align*}
We sum up the both sides of the inequality above and obtain
\begin{align*}
&\sum_{k=1}^\infty2^{-4kN}[u]_{\sigma;B^k}\le C\sum_{k=1}^{\infty}2^{3k-kN}\|u\|_{L_\infty}
+\frac{1}{2}\sum_{k=4}^\infty2^{-4kN}[u]_{\sigma;B^k}\\
&\quad +C\sum_{k=1}^\infty2^{-4kN+k}\sup_\beta\|f_\beta\|_{L_\infty(B_1)}
+C
\sum_{j=1}^\infty\big(\omega_u(2^{-j})+\omega_f(2^{-j})\big),
\end{align*}
which further implies that
\begin{align*}
\sum_{k=1}^\infty2^{-4kN}[u]_{\sigma;B^k}\le C\|u\|_{L_\infty}+C\sup_\beta\|f_\beta\|_{L_\infty(B_1)}
+C\sum_{j=1}^\infty\big(\omega_u(2^{-j})+\omega_f(2^{-j})\big),
\end{align*}
where $C$ depends on $d$, $\lambda$, $\Lambda$, $\sigma$, and $\omega_a$.
In particular, when $k=4$, we deduce
\begin{equation}
                            \label{eq12.09}
[u]_{\sigma;B^{4}}\le C\|u\|_{L_\infty}+C\sup_\beta\|f_\beta\|_{L_\infty(B_1)}
+C\sum_{j=1}^\infty\big(\omega_u(2^{-j})+\omega_f(2^{-j})\big),
\end{equation}
which apparently implies \eqref{eq12.17}.

Finally, since $\|v_1\|_{1}$ is bounded by the right-hand side \eqref{eq12.09}, from \eqref{eq12.03}, we see that
\begin{align*}
\sum_{j=0}^\infty 2^{j(\sigma-\alpha)}\sup_{x_0\in \bR^d}\inf_{p\in \cP_1}[v_1-p]_{\alpha;B_{2^{-j}}(x_0)}\le C.
\end{align*}
This and \eqref{eq9.06} with $u$ replaced by $v_1$ and $f_\beta$ replaced by $\eta_1f_\beta-h_{1\beta}$ give
\begin{align*}
&\sum_{j=k_1}^\infty 2^{(j+k_0)(\sigma-\alpha)}\sup_{x_0\in \bR^d}\inf_{p\in \cP_1}[v_{1}-p]_{\alpha;B_{2^{-j-k_0}}(x_0)}\\
&\le
C2^{-k_0\alpha}+C2^{k_0(\sigma-\alpha)}
\sum_{j=k_1}^\infty \big(\omega_f(2^{-j})+\omega_a(2^{-j})+\omega_u(2^{-j})+2^{-j\alpha}\big),
\end{align*}
Here we also used Lemma \ref{lem2.4} and \eqref{eq9.32} with $k=1$.
Therefore, for any small $\varepsilon>0$, we can find $k_0$ sufficiently large then $k_1$ sufficiently large, depending only on $C$, $\sigma$, $\alpha$, $\omega_f$, $\omega_a$, $\omega_f$, and $\omega_u$, such that
$$
\sum_{j=k_1}^\infty 2^{(j+k_0)(\sigma-\alpha)}\sup_{x_0\in \bR^d}\inf_{p\in \cP_1}[v_1-p]_{\alpha;B_{2^{-j-k_0}}(x_0)}<\varepsilon,
$$
which, together with the fact that $v_1 = u$ in $B_{1/2}$ and the proof of Lemma \ref{lem2.3} (ii), indicates that
$$
\sup_{x_0\in B_{1/2}} [u]_{\sigma;B_r(x_0)}\to 0 \quad\text{as}\quad r\to 0
$$
with a decay rate depending  only on $d$, $\lambda$, $\Lambda$, $\omega_a$, $\omega_f$, $\omega_u$, $\sup_{\beta\in \cA}\|f_\beta\|_{L_\infty(B_1)}$, and $\sigma$. Hence, the proof of the case when $\sigma\in (1,2)$ is completed.

{\bf The case when $\sigma = 1$.} The proof is very similar to the case when $\sigma\in (1,2)$ and we only provide a sketch here.  We use the same notation as in the previous case
\begin{equation*}
h_{k\beta}(x) = \int_{\bR^d}\frac{\xi_k(x,y) a_\beta(x,y)}{|y|^{d+1}}\,dy,
\end{equation*}
where
$$
\xi_k(x,y) := u(x+y)(\eta_k(x+y)-\eta_k(x))-u(x)y\cdot D\eta_k(x)\chi_{B_1}.
$$
It is easy to see that when $|y|\ge2^{-k-3}$,
$$|\xi_k(x,y)|\le C2^k|y|\|u\|_{L_\infty}.$$
On the other hand, when $|y|<2^{-k-3}$,
\begin{align*}
|\xi_k(x,y)| &\le |y|\int_0^1|u(x+y)D\eta_k(x+ty)-u(x)D\eta_k(x)|\,dt\\
&\le C2^{k}|y|w_u(|y|)+C2^{2k}|y|^2|u(x)|.
\end{align*}
 Therefore,
 \begin{align*}
 \|h_{k\beta}\|_{L_\infty}\le C2^k\Big(\|u\|_{L_\infty}+\int_0^1\frac{w_u(r)}{r}\,dr\Big).
 \end{align*}
Next we estimate the modulus of continuity of $h_{k\beta}$ and proceed as in the case when $\sigma\in (1,2)$.  Indeed, it is easily seen that
\begin{equation*}
{\rm II} \le C2^k\Big(\|u\|_{L_\infty}
+\int_0^1\frac{\omega_u(r)}{r}\,dr\Big)\,\omega_a(|x-x'|).
\end{equation*}
To estimate ${\rm I}$, we write
\begin{align*}
\xi_k(x,y)-\xi_k(x',y) = u(x+y)(\eta_k(x+y)-\eta_k(x))-u(x)y\cdot D\eta_k(x)\chi_{B_1}\\
-u(x'+y)(\eta_k(x'+y)-\eta_k(x'))+u(x')y\cdot D\eta_k(x')\chi_{B_1}.
\end{align*}
Obviously, when $|y|\ge2^{-k-3}$
\begin{align}
                                    \label{eq2.03}
|\xi_k(x,y)-\xi_k(x',y)|\le C2^{2k}|y|\big(\|u\|_{L_\infty}|x-x'| + \omega_u(|x-x'|)\big).
\end{align}
When $|y|<2^{-k-3}$, we have $\chi_{B_1}(y) = 1$. Thus similar to the first case,
\begin{align*}
&|\xi_k(x,y)-\xi_k(x',y)|\\
&\le |y|\int_0^1|(u(x+y)-u(x))D\eta_k(x+ty)-(u(x'+y)-u(x'))D\eta_k(x'+ty)|\,dt\\
&\quad +|y|^2\int_0^1\int_0^1|u(x)D^2\eta_k(x+tsy)-u(x')D^2\eta_k(x'+tsy)|\,dt\,ds := {\rm III}+{\rm IV}.
\end{align*}
Clearly,
\begin{align*}
{\rm IV}\le C2^{3k}|y|^2(\omega_u(|x-x'|)+\|u\|_{L_\infty}|x-x'|).
\end{align*}
When $x,x' \in (B^{k+2})^c$, we have ${\rm III}\equiv 0$.
When $x,x'\in B^{k+2}$, by the triangle inequality,
\begin{align*}
{\rm III}&\le |y|\int_{0}^1|u(x+y)-u(x)-(u(x'+y)-u(x'))||D\eta_k(x+ty)|\,dt\\
&\quad +|y|\int_0^1|u(x'+y)-u(x')||D\eta_k(x+ty)-D\eta_k(x'+ty)|\,dt\\
&\le C2^{k}|y|^{1+\gamma}|x-x'|^{\zeta}[u]_{\zeta+\gamma;B^{k+3}}
+C2^{2k}|y|\omega_u(|y|)|x-x'|,
\end{align*}
where $C$ depends on $d$, and $\zeta+\gamma<1$. Here we used the inequality
\begin{align*}
|u(x+y)-u(x)-(u(x'+y)-u(x'))|\le 2[u]_{\gamma+\zeta}|x-x'|^\zeta|y|^\gamma.
\end{align*}
Set $\gamma=\zeta = 1/4$.
When $x\in B^{k+2}$ and $x'\in (B^{k+2})^c$,
\begin{align*}
{\rm III}&=|y|\int_0^1|(u(x+y)-u(x))D\eta_k(x+ty)|\,dt\\
&=|y|\int_0^1|(u(x+y)-u(x))(D\eta_k(x+ty)-D\eta_k(x'+ty))|\,dt\\
&\le C2^{2k}|y|\omega_u(|y|)|x-x'|.
\end{align*}
The case when $x'\in B^{k+2}$ and $x\in (B^{k+2})^c$ is similar.
Then with the estimates of ${\rm III}$ and ${\rm IV}$ above, we obtain that when $|y|<2^{-k-3}$,
\begin{align*}
|\xi_k(x,y)-\xi_k(x',y)|\le C2^{3k}|y|^2\big(\omega_u(|x-x'|)+\|u\|_{L_\infty}|x-x'|\big)\\
+ C2^k|y|^{5/4}|x-x'|^{1/4}[u]_{1/2;B^{k+3}}+C2^{2k}|y|\omega_u(|y|)|x-x'|,
\end{align*}
which, combining with \eqref{eq2.03} for the case when $|y|\ge 2^{-k-3}$, further implies that
\begin{align*}
{\rm I}\le &C2^{2k}\Big(\omega_u(|x-x'|)+\|u\|_{L_\infty}|x-x'|\\
&\quad+[u]_{1/2;B^{k+3}}|x-x'|^{1/4}
+|x-x'|\int_0^1\frac{w_u(r)}{r}\,dr\Big),
\end{align*}
where $C$ depends on $d$ and $\Lambda$. Hence, we obtain the estimate of the modulus of continuity of $h_{k\beta}(x)$:
\begin{align*}
\omega_h(r)= C2^{2k}\Big(\omega_u(r)+[u]_{1/2;B^{k+3}}r^{1/4}
+\big(\|u\|_{L_\infty}+\int_0^1\frac{\omega_u(r)}{r}\,dr\big)
\big(r+\omega_a(r)\big)\Big).
\end{align*}
The rest of the proof is the same as the previous case. 
\end{proof}


\bibliographystyle{plain}
\def\cprime{$'$}

\end{document}